\documentclass[a4paper,12pt,oneside]{amsart}

\usepackage{amssymb}  
\usepackage{latexsym} 
\usepackage{comment}
\usepackage{color}
\usepackage{colonequals}

\usepackage[pdf,usenames,dvipsnames]{pstricks}
\usepackage{epsfig}
\usepackage{pst-grad} 
\usepackage{pst-plot} 

\definecolor{darkgreen}{rgb}{0,0.5,0}
\usepackage[
        colorlinks, citecolor=darkgreen,
        pdfauthor={Ingrid Bauer, Fabrizio Catanese},
        pdftitle={Rigidity of the Kummer covering},
        linktocpage        
]{hyperref}

\usepackage[alphabetic,backrefs,lite]{amsrefs} 

\usepackage[all]{xy} 

\usepackage{enumerate} 

\newcommand\sC{{\mathcal C}}

\newcommand\sE{{\mathcal E}}

\newcommand\sP{{\mathcal P}}

\newcommand\sI{{\mathcal I}}

\newcommand\sU{{\mathcal U}}
\newcommand\sL{{\mathcal L}}
\newcommand\sQ{{\mathcal Q}}

\newcommand\sN{{\mathcal N}}

\newcommand\sH{{\mathcal H}}

\newcommand\la{\lambda}

\newcommand\s{\sigma}

\newcommand\de{\delta}

\DeclareMathOperator{\grad}{grad}

\newcommand{\CC}{\ensuremath{\mathbb{C}}}
\newcommand{\RR}{\ensuremath{\mathbb{R}}}
\newcommand{\ZZ}{\ensuremath{\mathbb{Z}}}

\newcommand{\hol}{\ensuremath{\mathcal{O}}}

\newcommand{\BB}{\ensuremath{\mathbb{B}}}
\newcommand{\PP}{\ensuremath{\mathbb{P}}}

\newcommand{\ra}{\ensuremath{\rightarrow}}

\def\eea{\end{eqnarray*}}
\def\bea{\begin{eqnarray*}}

\newcommand\dual{\mathrel{\raise3pt\hbox{$\underline{\mathrm{\thinspace d
\thinspace}}$}}}
\newcommand\qe{\ifhmode\unskip\nobreak\fi\quad $\Box$}       

\def\BOX{\hfill\lower.5\baselineskip\hbox{$\Box$}}

\newtheorem{theorem}{Theorem}[section]
\newtheorem{lemma}[theorem]{Lemma}
\newtheorem{corollary}[theorem]{Corollary}
\newtheorem{proposition}[theorem]{Proposition}

\newtheorem{conjecture}[theorem]{Conjecture}

\newtheorem{question}[theorem]{Question}

\theoremstyle{remark}
\newtheorem{remark}[theorem]{Remark}

\theoremstyle{definition}
\newtheorem{definition}[theorem]{Definition}

\DeclareMathOperator{\rank}{rank}

\DeclareMathOperator{\GL}{GL}

\DeclareMathOperator{\im}{im}

\numberwithin{equation}{section}

 \newenvironment{dedication}
        {\begin{quotation}\begin{center}\begin{em}}
        {\par\end{em}\end{center}\end{quotation}}

\newcounter{nootje}
\renewcommand\check[1]
  {\marginpar{\tiny\begin{minipage}{20mm}\begin{flushleft}\thenootje : #1\end{flushleft}\end{minipage}}\addtocounter{nootje}{1}}
\begin{document}

\title[Rigid configurations and HK coverings]{Del Pezzo surfaces, rigid line configurations and Hirzebruch-Kummer coverings}

\author{Ingrid Bauer}
\address{Mathematisches Institut,
         Universit\"at Bayreuth,
         95440 Bayreuth, Germany.}
\email{Ingrid.Bauer@uni-bayreuth.de}

\author{Fabrizio Catanese}
\address{Mathematisches Institut,
         Universit\"at Bayreuth,
         95440 Bayreuth, Germany.}
\email{Fabrizio.Catanese@uni-bayreuth.de}
\date{\today}
\thanks{
The present work took place in the  framework  of  
the ERC-2013-Advanced Grant - 340258- TADMICAMT.  Part of the work was done while the authors were guests at KIAS,
the second author as KIAS Research scholar.}

 \makeatletter
    \def\@pnumwidth{2em}
  \makeatother

\maketitle
\begin{dedication}
In beloved  memory of Paolo (De Bartolomeis)
\end{dedication}
\addtocontents{toc}{\protect\setcounter{tocdepth}{1}}

\tableofcontents

\section{Introduction}

The answers to some long standing open questions in algebraic geometry (\cite{volmax}, \cite{cat-dettCR},\cite{cat-dett1},\cite{cat-dett2})
turned out to be related with a construction of Hirzebruch of Galois coverings $X$ of the projective plane $\PP^2 $ (we stick for simplicity
to the case of the complex projective plane $ \PP^2_{\CC}$ even if the results are more general)
with group $(\ZZ/n)^r$, branched over the union $\sL$ of $(r+1)$ lines (we call $\sL$ a configuration of lines since we want  the incidence relation of the
lines of $\sL$ to be fixed
\footnote{a common notation, by which $\sL$ is  called an {\em arrangement} of lines, allows  the incidence relation of the lines not to be  fixed:
in our case, when  we speak of a given configuration, we want a fixed incidence relation}).

The minimal resolution $S$ of the singularities of such a covering $X$ is called the Hirzebruch-Kummer covering 
of  the plane, of exponent $n$,  associated to $\sL$,
and $S$ is denoted by $HK (n, \sL)$ (the name `singular HK covering' is instead reserved for $X$).

The simplest interesting example occurs when $r=5$ and $\sL$ is the complete quadrangle $\sC \sQ$, 
the union of the sides and of the medians of
a triangle (in other words, the six lines joining pairs of points of a projective basis $P_1, P_2, P_3, P_4$).

In this case $S (n) : = HK (n, \sC \sQ)$ is a smooth ramified Galois covering of the Del Pezzo surface $Y_5$ of degree $5$,
the blow-up of the plane in the points $P_1, P_2, P_3, P_4$.

For $n=5$, $S (5) : = HK (5, \sC \sQ)$ is a ball quotient (cf. \cite{bhh}), in particular $S(5)$ enjoys the
following properties:

\begin{enumerate}
\item
$S(5)$ is rigid;
\item
$S(5)$ admits a Hermitian metric of strongly negative curvature;
\item
$S(5)$ is  a projective classifying space (indeed $S(5)$ has a contractible universal cover $\tilde{S}(5) \cong \BB_2 : = \{z \in \CC^2 | |z| < 1\}$);
\item
the universal cover $\tilde{S}(5)$ of $S(5)$ is Stein. 
\end{enumerate}

The first two properties are the most interesting ones, and especially in general the second property implies 
the remaining two ((3) follows indeed  from the existence
of a Riemannian metric with seminegative sectional curvature).  

A natural question (\cite{rigid}) is whether these properties extend, for exponent $n$ suffciently large,
 to Hirzebruch-Kummer  coverings $S (n) : = HK (n, \sL)$ 
associated to rigid line configurations $\sL$.

Indeed, the property of rigidity of $S  : = HK (n, \sL)$ clearly implies the rigidity of the line configuration $\sL$.
Among the rigid line configurations which were used by Hirzebruch to construct explicit surfaces which are ball quotients 
are the complete quadrangle, the Hesse $(9_4, 12_3)$ configuration of the $9$ flexes of a smooth cubic curve,
and the twelve lines joining pairs of flexpoints, and the dual Hesse $(12_3, 9_4)$ configuration.

Motivated by these and other considerations, we analysed in \cite{rigid} the converse question in the particular case of $\sC \sQ$,
establishing the following theorem:

\begin{theorem}\label{CQ}
The surface $ HK(n, \sC \sQ)$ is rigid (indeed, infinitesimally rigid)  if and only if $ n \geq 4$.

\end{theorem}

We raised therefore the following conjecture:

\begin{conjecture}
Given a rigid line configuration $\sL$, then the surface $ HK(n, \sL)$ is rigid for $n$ sufficiently large.
\end{conjecture}

The proof of theorem \ref{CQ} is quite long and technically involved, 
and makes use of the $\mathfrak S_5$-symmetries of the Del Pezzo surface $Y_5$ and vanishing theorems for
twisted sheaves of logarithmic forms. The proof does not use the deformation invariance of the fibrations
onto generalized Fermat curves induced by the projection of the plane with centre one of the singular points of the configuration.
Hence we outlined a strategy, whose first step is to show the rigidity of the topologically equisingular deformations
of the associated singular HK covering $ X $.

The first result of this paper is to establish the rigidity of the equisingular deformations of $X$, for a singular HK covering
branched on a line configuration which satisfies a saturation assumption, which is satisfied by the complete quadrangle,
by the iterated Burniat-Campedelli configurations, and by  the Hesse configuration of $12$
lines (the condition is not satisfied by the dual Hesse configuration, however a completely identical  proof
works out  in this case,  for the other  classical
rigid configurations mentioned above, and for many others).

\begin{theorem}
Let $\mathcal{L} = \{L_0, \ldots , L_r\}$ be a {\em singularly saturated} configuration of lines in $\PP^2$, i.e., such that
 it has  $m \geq 4$ singular points $p_1, \ldots , p_m$, and moreover:
\begin{enumerate}
\item each $L_i$ contains at least two singular points,
\item for $i\neq j \in \{1, \ldots , m\}$ there is a $k \in \{0, \ldots , r\}$ such that 
$$
\overline{p_ip_j} = L_k.
$$
\end{enumerate}
Then for each $n \geq 3$ the equisingular deformations of $X_{\mathcal{L}}(n)$ are infinitesimally trivial.
\end{theorem}

The proof of this  result is  given in  the  second section of the paper, while the first section is devoted to general properties 
of rigid line configuratios.

The last two sections  instead motivated by the second question,  and is restricted to the case of the line configuration
$\sC \sQ$. In this case, the existence of the desired metric is only known for the case where $5$ divides $n$:
this was done by Fangyang Zheng \cite{zheng}, extending a technique introduced by Mostow and Siu \cite{m-s}
for ramified coverings of ball quotients. In general, Panov \cite{panov} showed the existence of a non smooth 
negative metric (a polyhedral metric)
for $ n > n_0$, but where $n_0$ is unspecified. 

Here, our contribution is to establish explicit equations for  $ HK(n, \sC \sQ)$ as a submanifold of the product of
$C(n)^4$, where $C(n)$ is the Fermat curve of degree $n$. This is achieved through a new result: the description of the equations
of the Del Pezzo surface $Y_5$ as a submanifold of $(\PP^1)^4$.

The two main results here are (cf Theorems \ref{eqdp} and \ref{eqhk}):

 \begin{theorem}
 Let $\Sigma \subset (\PP^1)^4=:Q$, with coordinates 
 $$(v_1:v_2),(w_1:w_2),(z_1:z_2), (t_1:t_2),$$
  be the image of the Del Pezzo surface $Y$ via  $\varphi_1 \times \dots \times  \varphi_4$.
 Then the equations of $\Sigma$ are given by the four $3 \times 3$-minors of the following Hilbert-Burch matrix:
 \begin{equation}
 A:=
 \begin{pmatrix}
 t_2 & -t_1 &t_1+t_2\\
  v_1 & v_2 & 0\\
 w_2 & 0 &w_1 \\
 0&-z_1&z_2 
 \end{pmatrix} .
\end{equation}
In particular, we have a Hilbert-Burch resolution:
\begin{equation}
0 \ra (\hol_Q(-\sum_{i=1}^4 H_i))^{\oplus 3} \ra \bigoplus_{j=1}^4( \hol_Q(-\sum_{i=1}^4 H_i + H_j)) \ra \hol_Q \ra \hol_{\Sigma} \ra 0,
\end{equation}
where $H_i$ is the pullback to $Q$ of a point in $\PP^1$ under the i-th projection.
 \end{theorem}

\begin{theorem}
The equations of $S_n = HK(n,\sC\sQ) \subset C(n)^4$ are given by the four $3 \times 3$-minors of the following  matrix:
 \begin{equation}
 A':=
 \begin{pmatrix}
 T_2 & -T_1 &-T_3\\
  V_1 & V_2 & 0\\
 W_2 & 0 &W_1 \\
 0&-Z_1&Z_2 
 \end{pmatrix} ,
\end{equation}
and the linear syzygies among the four equations are given by the columns of the matrix $A'$.
\end{theorem}

We also show how from these equations one obtains the equations of the embedding of $Y_5$ in $(\PP^1)^5$,
and of the embedding $S_n = HK(n,\sC\sQ) \subset C(n)^5$, but we do not go into details here.
Also, we do not go  as far as to calculate the second fundamental form of these embeddings. 

We shall give more details  in a forthcoming  paper where we shall also describe the 
icosahedral symmetry of $Y_5$ (yielding geometric descriptions of the irreducible
$\mathfrak S_5$ representations), and the $\mathfrak S_5$-equivariant
Pfaffian equations of
the anticanonical embedding of $Y_5$.


\section{Infinitesimally rigid line configurations}

We consider a (planar) line configuration $\mathcal{L}$, consisting of $r+1$ lines $L_0, \ldots , L_r \subset \PP^2$,
and denote by $P_{i,j} : = L_i \cap L_j$. 
\begin{definition}
(1) We call $p = P_{i,j} $ a {\em singular point} of  the line configuration $\mathcal{L}:= \{L_0, \ldots L_r\}$, if $p$ has  valency $v_p \geq 3$,
i.e., the number $v_p$ of lines $L_k$  through $p$ is at least three.

(2) The variety of line configurations of type $\sL$ is the subvariety $\Sigma (\sL) \subset ((\PP^2)^{\vee})^{r+1}$ defined as
$$  \Sigma (\sL) : \{ (l_0, \dots, l_r) |  \cap_{i \in I} l_i \neq \emptyset \Leftrightarrow \cap_{i \in I} L_i \neq \emptyset, \forall I \subset \{ 0,1, \dots, r\}\}.$$
\end{definition}

Let $\pi \colon Y \rightarrow \PP^2$ be the blow-up of $\PP^2$ in the singular points $p_1, \ldots , p_m$ of $\mathcal{L}$, and denote by $E_i$ the exceptional curve over $p_i$. Let 
$$ 
\Delta_Y = \Delta := D_0 + \ldots +D_{r+m} \subset Y,
$$
where 
\begin{itemize}
\item $D_i :=$ the strict transform of $L_i$ under $\pi$ for $0 \leq i \leq r$,
\item $D_i := E_{i-r}$ for $r+1 \leq i \leq r+m$.
\end{itemize}
Then $\Delta$ is a reduced divisor with global normal crossings in $Y$.

\begin{definition} 
Observing that the group $G:= \PP\GL(3,\CC)$ acts on $\Sigma (\sL)$, one says that 

0) the line configuration  $\sL$ is said to be {\em projectively unique} iff $\Sigma (\sL)$
coincides with  the $G$-orbit of $\sL$;

1) the line configuration $\sL$ is called {\em rigid} if $\Sigma (\sL)$
 and   the $G$-orbit of $\sL$ coincide set theoretically  in a neighbourhood of $\sL$;
 
 2) the line configuration $\sL$ is called {\em infinitesimally rigid}, if 
$\Sigma (\sL)$
 and   the $G$-orbit of $\sL$ coincide scheme theoretically  in a neighbourhood of $\sL$.
 
 Equivalently, in view of Kodaira's theorem on the stability of $(-1)$ curves by deformation, \cite{kodaira},

1') the line configuration $\sL$ is called {\em rigid}, if the pair $(Y, \Delta)$ is rigid;

2') the line configuration $\sL$ is called {\em infinitesimally rigid}, if 
$$
H^1(Y, \Omega_Y^1( \log \Delta)(K_Y)) ( \cong H^1(Y, \Theta_Y( - \log \Delta))^*) = 0.
$$
\end{definition}

\begin{remark}
Let $\mathcal{L}$ be a line configuration and denote by $n_i$ the number of singular points on $L_i$. Then
if $\sL$ is rigid,  we shall see that 
necessarily $n_i \geq 2$ for all $i$ (else we could move  the line $L_i$ in a pencil). In particular, all the curves $D_i$
have negative self-intersection.

\begin{itemize}
\item $D_i^2 = 1 - n_i$ and $D_iK_Y = n_i -3$ for $0 \leq i \leq r$,
\item $D_i^2 = -1 $ and  $D_iK_Y = -1$ for $r+1 \leq i \leq r+m$, so that these $D_i$'s are $(-1)$-curves.
\end{itemize}
\end{remark}
\begin{lemma}
Let $\mathcal{L}$ be a line configuration. Assume that $m \geq 4$ and that $n_i \geq 2$ for all $0 \leq i \leq r$. Then
\begin{equation}
\chi(\Omega_Y^1( \log \Delta)(K_Y)) = 8 - 2m + \sum_{i=0}^r (n_i - 2).
\end{equation}
\end{lemma}

\begin{proof}
Consider the exact sequence
\begin{equation}\label{seq}
0 \rightarrow \Omega_Y^1(K_Y) \rightarrow \Omega_Y^1( \log \Delta)(K_Y) \rightarrow \bigoplus_{i=0}^{r+m} \mathcal{O}_{D_i}(K_Y) \rightarrow 0.
\end{equation}
This imples that 
\begin{multline}
\chi(\Omega_Y^1( \log \Delta)(K_Y)) = \chi(\Omega_Y^1(K_Y)) + \sum_{i=0}^{r+m} \chi(\mathcal{O}_{D_i}(K_Y)) = \\
= \chi(\Theta_Y) + \sum_{i=0}^{r+m} \chi(\mathcal{O}_{\PP^1}(K_Y \cdot D_i)) = 2K_Y^2 - 10\chi(\mathcal{O}_Y) + \sum_{i=0}^{r+m}(1+K_Y \cdot D_i) = \\
= 8 - 2m + \sum_{i=0}^r (n_i - 2).
\end{multline}
\end{proof}

\begin{remark}
If $\mathcal{L}$ is infinitesimally rigid, then $\chi(\Omega_Y^1( \log \Delta)(K_Y))  \geq 0$, i.e.
\begin{equation}
2m-8 \leq \sum_{i=0}^r (n_i - 2).
\end{equation}\label{ineq}
In particular, if $m \geq 5$, then any infinitesimally rigid  line configuration has lines containing more than three singular points.

Moreover,  if $\mathcal{L}$ is rigid, then $m \geq 4$, hence the set of singular points contains a projective basis,
hence   $n_i \geq 2$ for each $i$ (otherwise we can move the line $L_i$ in a pencil).
\end{remark}

Consider the long exact cohomology sequence associated to (\ref{seq})
\begin{multline}\label{cohseq}
0 \rightarrow H^0(\Omega_Y^1(K_Y)) = 0 \rightarrow H^0(\Omega_Y^1( \log \Delta)(K_Y)) \rightarrow \bigoplus_{i=0}^{r} H^0(\mathcal{O}_{D_i}(K_Y)) \cong\\
\cong \bigoplus_{i=0}^{r} H^0 (\mathcal{O}_{\PP^1}(n_i-3)) \rightarrow H^1(\Omega_Y^1(K_Y))  \cong \CC^{2m-8} \rightarrow H^1(\Omega_Y^1( \log \Delta)(K_Y)) \rightarrow 0.
\end{multline}

\begin{remark} \

1) It follows from the continuation of the above long exact cohomology sequence that $H^2(\Omega_Y^1( \log \Delta)(K_Y)) \cong H^2(\Omega_Y^1(K_Y)) \cong H^0(Y, \Theta_Y)^* =0$, if the singular points contain a projective basis of $\PP^2$.

2) The following proposition, whose proof follows immediately from (\ref{cohseq}),  gives a sufficient and necessary condition for the infinitesimal rigidity of a line configuration.
\end{remark}

\begin{proposition}
Let $\mathcal{L}$ be a line configuration. Assume that $m \geq 4$ and that $n_i \geq 2$ for all $0 \leq i \leq r$. Then the following are equivalent:

\begin{enumerate}
\item $\mathcal{L}$ is infinitesimally rigid,
\item
$\delta \colon \bigoplus_{i=0}^{r} H^0 (\mathcal{O}_{\PP^1}(n_i-3)) \rightarrow H^1(\Omega_Y^1(K_Y))$ 
is surjective,
\item $\dim H^0(\Omega_Y^1(K_Y)) = 8-2m +\sum_{i=0}^r (n_i -2)$.
\end{enumerate}
\end{proposition}

This leads naturally to the following:
\begin{question}
Let $\mathcal{L}$ be a line configuration such that
\begin{itemize}
\item there are $m\geq 4$ singular points,
\item each line contains at least $2$ singular points, i.e., $n_i \geq 2$ for all $i$,
\item $2m-8 \leq \sum_{i=0}^r (n_i - 2)$.
\end{itemize}
Is then $\mathcal{L}$ rigid?

 If not, give combinatorial conditions which ensure the infinitesimal rigidity of $\mathcal{L}$.
\end{question}

\begin{remark}
It is easy to verify that the complete quadrangle, the Hesse configuration and the dual Hesse configuration are infinitesimally rigid. 
\end{remark}

Once we have infinitesimal rigid line configurations it is easy to construct new infinitesimally rigid line configurations as the following proposition shows:

\begin{proposition}\label{newrigid}
Assume that $\mathcal{L} = \{L_0, \ldots , L_r\}$ is a line configuration with singular points $p_1, \ldots , p_m$, such that $m \geq 4$ and $n_i \geq 2$ for all $0 \leq i \leq r$. Let $L_{r+1} \notin \mathcal{L}$ be a line passing either  through
\begin{enumerate}
\item[a)] exactly two double points $p_{m+1}, p_{m+2}$ of $\mathcal{L}$, or
\item[b)] $p_1$ and exactly one double point $p_{m+1}$ of $\mathcal{L}$, or
\item[c)] $p_1$ and $p_2$ (and no double points).
\end{enumerate}
Let $\mathcal{L}' := \{L_0, \ldots , L_r, L_{r+1}\}$ be the line configuration obtained from  $\mathcal{L}$ by adding the line $L_{r+1}$. Let $Y'$ be the blow up of $\PP^2$ in the singular points of $\mathcal{L}'$ and let $\Delta'$ be the divisor on $Y'$ consisting of the exceptional curves and the strict transforms of the lines. Then 
\begin{enumerate}
\item $\mathcal{L}'$ has $r+2$ lines and
\begin{enumerate}
\item[a)] $m'=m+2$,
\item[b)] $m'=m+1$,
\item[c)] $m'=m$.
\end{enumerate}
\item If $L_{r+1}$ does not contain any further singular point, then 
$$
\chi(\Omega_Y^1( \log \Delta)(K_Y)) = \chi(\Omega_{Y'}^1( \log \Delta ')(K_{Y'})).
$$
In general,
$$
\chi(\Omega_Y^1( \log \Delta)(K_Y)) \leq \chi(\Omega_{Y'}^1( \log \Delta ')(K_{Y'})).
$$

\item If $L_{r+1}$ does not contain any further singular point, then  $\mathcal{L}$ is infinitesimally rigid if and only if $\mathcal{L'}$ is infinitesimally rigid.
\end{enumerate}
\end{proposition}

\begin{proof}
1) is obvious.

2) Observe that the right hand side of (\ref{ineq}), $\sum (n_i-2)$, goes up by at least $4$ in case a), by at least $2$ in case b), and at least $0$ in case c). Therefore in all cases $\chi(\Omega_Y^1( \log \Delta)(K_Y)) \leq \chi(\Omega_{Y'}^1( \log \Delta ')(K_{Y'}))$, with equality  in cases a) and c)
exactly  if $L_{r+1}$ does not contain any further singular point.

3) a) Let $p \colon Y' \rightarrow Y$ be the blow-up of $Y$ in the points $p_{m+1}, p_{m+2}$ (to be precise: the inverse images of $p_{m+1}, p_{m+2} \in \PP^2$ in $Y$) with exceptional curves $E_{m+1}, E_{m+2}$. Then 
$$ 
\Delta ' = \Delta_{Y'} = p^*\Delta_Y + L_{r+1}= D_0 + \ldots +D_{r+m} + E_{m+1} + E_{m+2} + L_{r+1} \subset Y'.
$$
Note that by a slight abuse of notation we write $D_i$ for the irreducible divisor of $Y$ as well as for its strict transform on $Y'$.

We shall show that $h^1(\Omega_Y^1( \log \Delta)(K_Y)) = h^1(\Omega_{Y'}^1( \log \Delta ')(K_{Y'}))$, which implies the claim.

In fact, since $\chi(\Omega_Y^1( \log \Delta)(K_Y)) = \chi(\Omega_{Y'}^1( \log \Delta ')(K_{Y'}))$, and 
$$
h^2(\Omega_Y^1( \log \Delta)(K_Y)) = h^2(\Omega_{Y'}^1( \log \Delta ')(K_{Y'})) = 0,
$$
 we see that 
\begin{multline}
h^1(\Omega_Y^1( \log \Delta)(K_Y)) = h^1(\Omega_{Y'}^1( \log \Delta ')(K_{Y'}))  \iff  \\
\iff h^0(\Omega_Y^1( \log \Delta)(K_Y)) = h^0(\Omega_{Y'}^1( \log \Delta ')(K_{Y'}))
\end{multline}
Consider the exact sequence 
\begin{multline}
0 \rightarrow \Omega_{Y'}^1( \log (\Delta' - L_{r+1} - E_{m+1}- E_{m+2})) (K_{Y'}) \rightarrow \Omega_{Y'}^1( \log \Delta ')(K_{Y'}) \rightarrow \\
\rightarrow  \mathcal{O}_{L_{r+1}}(K_{Y'}) \oplus \mathcal{O}_{E_{m+1}}(K_{Y'}) \oplus \mathcal{O}_{E_{m+2}}(K_{Y'}) \cong \mathcal{O}_{\PP^1}(-1)^{\oplus 3} \rightarrow 0
\end{multline}
Therefore $h^0(\Omega_{Y'}^1( \log \Delta ')(K_{Y'})) = h^0( \Omega_{Y'}^1( \log \tilde{\Delta} )(K_{Y'}))$, where 
$$
\tilde{\Delta} := \Delta' - L_{r+1} - E_{m+1}- E_{m+2}.
$$

Observe that 
\begin{multline}
p_*( \Omega_{Y'}^1( \log \tilde{\Delta} )(K_{Y'})) = p_*( \Omega_{Y'}^1( \log \tilde{\Delta} )(E_{m+1}+E_{m+2}) \otimes p^*\mathcal{O}_Y(K_{Y})) \\
\cong p_* \Omega_{Y'}^1( \log \tilde{\Delta} )(E_{m+1}+E_{m+2}) \otimes \mathcal{O}_Y(K_{Y}) \cong \Omega_Y^1( \log \Delta)(K_Y),
\end{multline}
where the last isomorphism holds by \cite{burniat3}, lemma (7.1), 4).

Therefore we have seen that 
$$
H^0(\Omega_{Y'}^1( \log \Delta ')(K_{Y'})) =H^0(p_*(\Omega_{Y'}^1( \log \Delta ')(K_{Y'}))) = H^0(\Omega_Y^1( \log \Delta)(K_Y)).
$$

b), c) are proven in the same way.
\end{proof}

\begin{remark} The above proposition shows that the iterated Burniat-Campedelli configurations introduced in \cite{rigid}
(and many other configurations) are infinitesimally rigid.
\end{remark}


\section{Equisingular deformations of complete intersections}

\subsection{Equisingular deformations of complete intersection singularities} \

Recall the notion of equisingular deformations of a subvariety $X \subset Y$, where $Y$ is smooth and $X$ has only isolated 
singularities:

\begin{definition} \
Let $Y$ be a smooth complex  manifold and   $X \subset Y$ an analytic subspace with  only isolated singularities.

Consider the exact sequence (where $\sI_X$ is the ideal sheaf of $X$ in $Y$)
\begin{equation}
0 \rightarrow \mathcal{N}_{X|Y}^{\vee} : = \sI_X/ \sI_X^2 \rightarrow \Omega^1_Y \otimes \mathcal{O}_X \rightarrow \Omega^1_X \rightarrow 0,
\end{equation}

and dualize it (apply the functor $\sH om _{\hol_X} ( - , \hol_X)$) to obtain the long exact sequence 

\begin{equation}
0 \rightarrow \Theta_X  \ra \Theta_Y \otimes \mathcal{O}_X \rightarrow \mathcal{N}_{X|Y} \rightarrow \mathcal{E}xt^1_{\mathcal{O}_X}( \Omega^1_X, \mathcal{O}_X) \rightarrow 0
\end{equation}

1) We define as usual the {\em equisingular normal sheaf} $\mathcal{N}'_{X|Y}$ of $ X \subset Y$ as
\begin{equation}
\mathcal{N}'_{X|Y}:= \im(\delta) = \ker(\mathcal{N}_{X|Y} \rightarrow \mathcal{E}xt^1_{\mathcal{O}_X}( \Omega^1_X, \mathcal{O}_X)),
\end{equation}

so that we have an exact cohomology sequence 

$$ 0 \ra H^0(\Theta_X )  \ra H^0( \Theta_Y \otimes \hol_X ) \ra H^0(\mathcal{N}'_{X|Y} ) \ra H^1(\Theta_X )  \ra H^1( \Theta_Y \otimes \hol_X ) \ra 0.$$

\end{definition}

Assume now  that $(X,0) \subset (Y,0) :=(\CC^k, 0) $ is a complete intersection singularity of dimension $2$, given by equations $f_3, \ldots , f_k \in \CC [z_1, \ldots, z_k]$ of respective degrees $d_i$, $3 \leq i \leq k$.

The above  exact sequence reads out as:
\begin{equation}
0 \rightarrow \mathcal{N}_{X|Y}^{\vee} \cong \bigoplus_{i=3}^k \mathcal{O}_X(-d_i) \rightarrow \Omega^1_Y \otimes \mathcal{O}_X \rightarrow \Omega^1_X \rightarrow 0,
\end{equation}

where $\delta ^{\vee} \colon \mathcal{N}_{X|Y}^{\vee} \cong \bigoplus_{i=3}^k \mathcal{O}_X(-d_i) \rightarrow \Omega^1_Y \otimes \mathcal{O}_X$ is given by 
$$
\delta ^{\vee}([f_i]) = df_i = \sum_{j=3}^k \frac{\partial f_i}{\partial z_j} dz_i.
$$

In the dual  sequence 
\begin{equation}
0 \ra \Theta_X  \ra \Theta_Y \otimes \hol_X \ra \mathcal{N}_{X|Y} \rightarrow \mathcal{E}xt^1_{\hol_X}( \Omega^1_X, \hol_X) \ra 0
\end{equation}

 $\delta$ is given by the matrix
\begin{equation}
\delta:
= 
\begin{pmatrix}
\grad (f_3)\\
\grad (f_4)\\
\cdots \\
\grad (f_k)
\end{pmatrix}
= 
\begin{pmatrix}
\frac{\partial f_3}{\partial z_1} & \frac{\partial f_3}{\partial z_2} &\ldots &\frac{\partial f_3}{\partial z_k} \\
\frac{\partial f_4}{\partial z_1} & \frac{\partial f_4}{\partial z_2} &\ldots &\frac{\partial f_4}{\partial z_k} \\
\cdots & && \cdots \\
\frac{\partial f_k}{\partial z_1} & \frac{\partial f_k}{\partial z_2} &\ldots &\frac{\partial f_k}{\partial z_k} \\
\end{pmatrix}
.
\end{equation}

According to the above definition, 

2) the {\em first order equisingular deformations of} the complete intersection singularity $(X,0) \subset \CC^k$ are  the complete intersections 
$$
\tilde{X} := \{z\in \CC^k : \tilde{f}_3(z)= \ldots = \tilde{f}_k(z) = 0\},
$$
where 
$$
\tilde{f}_j(z) = f_j(z) + \epsilon \varphi_j(z), \ \epsilon^2 =0, \ 3 \leq j \leq k, \  (\varphi_j) \in  \sN_{X|Y}' = \im(\delta).
$$

We consider now the following special case of complete intersection singularities: let $l_1, \ldots , l_k \in \CC[x_1,x_2]$ be distinct linear forms
(hence vanishing at the origin). Without loss of generality we can assume that 
\begin{equation}
l_1 = x_1, \ l_2 = x_2, \ l_j = \alpha_j x_1 + \beta_j x_2, \ 3 \leq j \leq k.
\end{equation}
Then we consider the complete intersection $X \subset \CC^k$ given by 
$$
f_j(z_1, \ldots , z_k) = \alpha_jz_1^n + \beta_jz_2^n - z_j^n, \ 3 \leq j \leq k. 
$$

Observe that these are the local equations of the singular Hirzebruch-Kummer covering of exponent $n$ over a singular point $p$ of the line configuration, where $k$ distinct lines meet.

Here:
\begin{equation}
\frac{1}{n} \grad(f_j) = (\alpha_jz_1^{n-1}, \beta_jz_2^{n-1}, 0, 0 \ldots ,0,  -z_j^{n-1}, 0 , \ldots, 0),
\end{equation}

whence the infinitesimal equisingular deformations of $(X,0)$ are the complete intersections given by 
$\tilde{f}_j(z) = f_j(z) + \epsilon \varphi_j(z)$, where 
\begin{equation}\label{local}
\varphi_j(z) = \alpha_jz_1^{n-1}u_1(z) + \beta_jz_2^{n-1}u_2(z) -z_j^{n-1}u_j(z).
\end{equation}

\subsection{Equisingular deformations of Kummer coverings} \

We consider $r+1$ different lines $L_0, \ldots L_r \subset \PP^2$, not all passing through the same point.

Set  $L_i = \{l_i = 0\}$, and assume, without loss of generality,   that 
\begin{align*}
l_0(x_0:x_1:x_2) =x_0,\\
l_1(x_0:x_1:x_2) =x_1,\\
l_2(x_0:x_1:x_2) =x_2.
\end{align*}

Consider the cartesian diagram:
\begin{equation}
\xymatrix{
X_{\mathcal{L}}(n)  \ar[d] \ar@^{(->}[r] & \PP^r, \ar[d] & z_i \ar@{|->}[d] \\
\PP^2 \ar@^{(->}[r]_{(l_0:\ldots :l_k)} & \PP^r, & z_i^n .\\
}
\end{equation}

Then 
\begin{equation}
X_{\mathcal{L}}(n) = \{ z \in \PP^r : F_j(z) =  l_j(z_0^n:z_1^n:z_2^n) - z_j^n = 0, 3 \leq j \leq r\}
\end{equation}
is a complete intersection in $\PP^r$.

Observe that $X_{\mathcal{L}}(n)$ is nonsingular if and only if the configuration $\mathcal{L}$ does not have singular points.

\begin{definition}
Let $\mathcal{L}$  be a line configuration in $\PP^2$, such that not all lines meet in one point. Then $X_{\mathcal{L}}(n)$ is called the {\em singular Kummer cover of exponent $n$ of $\PP^2$ branched on $\mathcal{L}$}.
\end{definition}
If the line configuration $\mathcal{L}$ is fixed we sometimes denote $X_{\mathcal{L}}(n)$ by $X(n)$.

For the (global) complete intersection  $X=X_{\mathcal{L}}(n) \subset \PP^r$ we have the following two exact sequences of coherent sheaves:

\begin{equation} \label{ex1}
0 \ra \Theta_X \ra \Theta_{\PP^r} \otimes \hol_X \ra \mathcal{N}'_X \ra 0 ,
\end{equation}

\begin{equation} \label{ex2}
0 \ra \mathcal{N}'_X \ra  \mathcal{N}_X  \ra \mathcal{E}xt^1_{\hol_X}(\Omega^1_X,\hol_X) \ra 0.
\end{equation}

The long exact cohomology sequence associated to \ref{ex1} gives:

$$
... \ra H^0(X, \Theta_{\PP^r} \otimes \hol_X ) \ra  H^0(X, \mathcal{N}'_X) \ra H^1(X,  \Theta_X) \ra H^1(X,  \Theta_{\PP^r} \otimes \hol_X ) \ra .... .
$$

 Sernesi proved that, if $ dim (X) \geq 2$, and unless $ dim (X)= 2, \hol_X(K_X) \cong \hol_X$, all small 
 deformations of complete intersections are complete intersections (cf. \cite{sernesi}, \cite{montecatini}), 
 in particular follows  that $H^1(X,  \Theta_{\PP^r} \otimes \hol_X ) = 0$.

Observe that 

$\Phi \in H^0(X, \mathcal{N}'_X)$ maps to zero in $H^1(X,  \Theta_X)$, i.e., gives a trivial deformation if and only if 
$$
\Phi \in \im(H^0(X, \Theta_{\PP^r} \otimes \hol_X ) \ra  H^0(X, \mathcal{N}'_X)).
$$

By the Euler sequence restricted to $X$
$$
0 \ra \hol_X \ra \hol_X (1)^{r+1} \ra \Theta_{\PP^r} \otimes \hol_X \ra 0,
$$
and since $H^1(X, \hol_X) = 0$ (X being a complete intersection), we conclude that 
$H^0(X, \hol_X (1)^{r+1}) \twoheadrightarrow H^0(X,\Theta_{\PP^r} \otimes \hol_X)$ is surjective. Composing with $H^0(X, \Theta_{\PP^r} \otimes \hol_X) \rightarrow H^0(X,\mathcal{N}'_X)$, we get a map
$$
H^0(X, \hol_X (1))^{r+1} \rightarrow H^0(X,\mathcal{N}'_X) \subset H^0(X,\mathcal{N}_X) = H^0(X, \hol_X (n))^{r-2},
$$
$$
(a_0, \ldots,a_r) \mapsto (\sum_{k=0}^r \frac{\partial F_3}{\partial z_k} a_k, \ldots , \sum_{k=0}^r \frac{\partial F_r}{\partial z_k} a_k).
$$

Therefore we have proven:

\begin{lemma}
$\Phi = (\Phi_3, \ldots , \Phi_r) \in H^0(X,\mathcal{N}'_X)$ gives a trivial deformation (i.e., maps to zero in $H^1(X,  \Theta_X)$) if and only if there are linear forms $(a_0, \ldots,a_r) \in H^0(X, \hol_X (1))^{r+1}$ such that
$$
\Phi_j = \sum_{k=0}^r \frac{\partial F_j}{\partial z_k} a_k, \ \ 3 \leq j \leq r.
$$
\end{lemma}

Assume now that $p$ is a singular point of the line configuration $L_0, \ldots , L_r$, so that
$$
\{p\} \in L_i \cap L_j \cap L_k.
$$
Then there exist $\la, \mu \in \CC$ with
$$
l_j= \lambda l_i + \mu l_k, 
$$
and we get one   equation vanishing at  the singular points of $X(n)$ mapping to $p$:
$$
z_j^n = \lambda z_i^n + \mu z_k^n.
$$
Then 

\begin{multline}
F_j = z_j^n - l_j(z_0^n, z_1^n,z_2^n) = z_j^n - (\lambda l_i + \mu l_k)(z_0^n, z_1^n,z_2^n) = \\
=\lambda F_i + \mu F_k +z_j^n -\lambda z_i^n -  \mu z_k^n. \\
\implies  F_j -\lambda F_i + \mu F_k = z_j^n -\lambda z_i^n -  \mu z_k^n.
\end{multline}

Therefore, by the hypothesis of equisingularity,   there are linear forms $u,v,w$ (depending on $p$) such that

\begin{equation}\label{relationphi}
 \Phi_j -\lambda \Phi_i + \mu \Phi_k = uz_j^{n-1} -\lambda vz_i^{n-1} -  \mu wz_k^{n-1}.
\end{equation}

We shall prove now the following

\begin{theorem}\label{rigid}
Let $\mathcal{L} = \{L_0, \ldots , L_r\}$ be a {\em singularly saturated} configuration of lines in $\PP^2$, i.e., such that
 it has  $m \geq 4$ singular points $p_1, \ldots , p_m$, and moreover:
\begin{enumerate}
\item each $L_i$ contains at least two singular points,
\item for $i\neq j \in \{1, \ldots , m\}$ there is a $k \in \{0, \ldots , r\}$ such that 
$$
\overline{p_ip_j} = L_k.
$$
\end{enumerate}
Then for each $n \geq 3$ the equisingular deformations of $X : = X_{\mathcal{L}}(n)$ are infinitesimally trivial (i.e., $H^1 (X, \Theta_X)= 0$).
\end{theorem}

\begin{remark}
Assumption 2) of the above theorem can be weakened. One has to find a good inductive condition which replaces 2). This will be explained in the proof of the theorem.
\end{remark}
\begin{proof}
We can assume without loss of generality that $p_1, p_2, p_3$ are not collinear, and moreover, we can assume that they are the coordinate points
$p_1 =(1:0:0)$, $p_2=(0:1:0)$, $p_3=(0:0:1)$, hence
\begin{align*}
\overline{p_1p_2} &= L_2= \{x_2 =0\},\\
\overline{p_1p_3} &= L_1= \{x_1 =0\},\\
\overline{p_2p_3} &= L_0= \{x_0 =0\};
\end{align*}
Observe that by our assumptions there is another singular point $p$ of $\mathcal{L}$, which is not contained in $L_0 \cup L_1 \cup L_2$, say $p=p_4 = (1:1:1)$.

We can assume without loss of generality that
$$
\overline{p_ip_4} = L_{2+i}, \ \ 1 \leq i \leq 3,
$$
more precisely we assume that 
\begin{equation}\label{l3-5}
l_3 = x_1-x_2, \ l_4 = x_2 - x_0, \ l_5 = x_0 - x_1.
\end{equation}
We consider  the equisingular deformations of $X_{\mathcal{L}}(n)$ given by the equations 
$\tilde{F}_j(z) = F_j(z)+ \Phi_j(z)$, where $\Phi = (\Phi_3, \ldots , \Phi_r) \in H^0(X,\mathcal{N}'_X)$, in particular $\Phi_j(z)$ is a homogeneous polynomial of degree $n$.

By (\ref{l3-5}) and (\ref{local}) we infer the existence of  linear forms $u_i$ (depending on the point $p_1$), $v_j$  and $w_k $
(depending upon  $p_2$, respectively $p_3$),  such that (here and in the following) on $X$

\begin{equation}\label{1}
\Phi_3 = u_3z_3^{n-1} -u_1z_1^{n-1} +u_2z_2^{n-1}, 
\end{equation}
\begin{equation}\label{2}
\Phi_4 = v_4z_4^{n-1} -v_2z_2^{n-1} +v_0z_0^{n-1},
\end{equation}
\begin{equation}\label{3}
\Phi_5 = w_5z_5^{n-1} -w_0z_0^{n-1} +w_1z_1^{n-1}.
\end{equation}

Since $l_3+ l_4 +l_5=0$, we get by (\ref{relationphi}) that there are linear forms $a_3, a_4, a_5$ such that

\begin{equation}\label{sum}
\Phi_3 + \Phi_4 +  \Phi_5 = a_3z_3^{n-1}+a_4z_4^{n-1}+a_5z_5^{n-1}.
\end{equation}
Comparing this to equations (\ref{1} - \ref{3}) we see that, as an easy calculation shows, 
$$
a_3=u_3, \ a_4= v_4, \ a_5=w_5, 
$$
and 
$$
a_0:=v_0=w_0, \ a_1:=u_1=w_1, \ a_2:= u_2=w_2.
$$
Replacing $\Phi_j$ by $\Phi_j - \sum_{k=0}^5 a_k\frac{\partial F_j}{\partial z_k}$ (i.e., with a globally trivial deformation) we can assume 
$$
\Phi_3 = \Phi_4 =\Phi_5 = 0.
$$

Let now $L_j =\{l_j=0\}$ be a further line through a coordinate point. Without loss of generality
we may assume that  $p_1 \in L_j$ (and $L_j \neq L_0, L_1, L_3$). Then we may write $l_j = \lambda x_1 + \mu x_2$ and we get that there are linear forms $u_j, u_3, u_2, u_1$ (depending on $p_1$) such that
$$
\Phi_j = u_jz_j^{n-1}- \lambda u_1 z^{n-1} -\mu u_2 z_2^{n-1}, 
$$
and 
$$
\Phi_3 = u_3 z_3^{n-1} -u_1z_1^{n-1} +u_2z_2^{n-1}.
$$
Since $\Phi_3 = 0$, we see easily that 
in particular $u_1 = u_2 =0$, hence $\Phi_j = u_jz_j^{n-1}$. Therefore after a globally trivial deformation we may assume $\Phi_j = 0$.

A similar argument shows that we can assume $\Phi_j = 0$ for each line $L_j$ through $p_4$. 

Let now 
 $$
 \sL' := \{ L \in \sL | \ \exists \  i \in \{1,2,3,4\}  : p_i \in L \}.
 $$
 
 Then after a globally trivial deformation we may assume that $\Phi_k=0$ for each $k$ such that $L_k \in \sL'$.

Consider now a line $L_j \in \sL$ such that $p_1,p_2,p_3,p_4 \notin L_j$. By assumption (1) $L_j$ contains two singular points $p,p'$, and by assumption (2) the lines  $\overline{p_ip}$ for $i=1,2,3$ are in $\sL$. Note that the three lines $\overline{p_1p}, \overline{p_2p}, \overline{p_3p}$ are not all equal, so without loss of generality 
$$
L_i := \overline{p_1p} \neq L_k:= \overline{p_2p}.
$$
Moreover, $\overline{p_ip} \neq \overline{p_ip'}$ for all $i = 1,2,3$, since otherwise $p_i \in L_j = \overline{pp'}$, a contradiction.

Hence we have four distinct lines $L_i,L_k,L_h,L_s \in \mathcal{L}'$ such that $p \in L_i, L_k$, $p' \in L_h,L_s$.
We can now write
$$
l_j = al_i + bl_k = cl_h +dl_s, \ a,b,c,d \in \CC.
$$
Using that $\Phi_i =  \Phi_k =\Phi_h =\Phi_s=0$ and applying (\ref{relationphi}) we get the existence of linear forms
$u_j, u_i, u_k, v_j , v_h, v_s$ yielding  the following  two equations for $\Phi_j$:
 
$$
\Phi_j = u_j z_j^{n-1} -au_iz_i^{n-1} - bu_kz_k^{n-1},
$$
$$
\Phi_j = v_j z_j^{n-1} -cv_hz_h^{n-1} - dv_sz_s^{n-1}.
$$
Comparing these two expressions for $\Phi_j$ we easily obtain  that 
$$
v_h=v_s=u_i=u_k =0, \ \Phi_j = v_j z_j^{n-1},
$$

and after a globally trivial deformation we achieve $\Phi_j = 0$.

This proves the theorem.

\end{proof}

\begin{remark}
 The inductive assumption that we need is:
\begin{itemize}
\item there is a chain $\sL' \subset \sL'_1 \subset \sL'_2 \subset \dots \subset \sL' _t \subset \sL' _{t+1} = \sL$
such that  $ \sL' _{j+1} \supset \sL' _j  \cup \{ L_{j+1}\}$ with  $L_{j+1} \notin \sL' _j$
\item there are  two singular points $p, p' \in L_{j+1}$ and
 four distinct lines $L_i,L_k,L_h,L_s \in \sL'_j$ such that $p \in L_i, L_k$, $p' \in L_h,L_s$
 \item 
 $ \sL' _{j+1}$ contains all the lines through $p, p'$.
\end{itemize} 
\end{remark}

The second hypothesis can be replaced by:
\begin{itemize}
\item there are three distinct lines $L_i,L_k,L_h \in \mathcal{L}'$ such that $p \in L_i, L_k, L_h$.
\end{itemize}

\begin{remark}
Observe that the complete quadrangle, the Hesse configuration  satisfy the assumptions of Theorem \ref{rigid}, while the dual Hesse configuration only satisfies the inductive assumption.
\end{remark}

\begin{corollary}
Assume that $\mathcal{L} = \{L_0, \ldots , L_r\}$ is a line configuration, such that  the equisingular deformations of $X_{\mathcal{L}}(n)$ are infinitesimally trivial.
 Let $L_{r+1} \notin \mathcal{L}$ be a line passing  through two points of  $\mathcal{L}$  of multiplicity at least 2.
 
Let $\mathcal{L}' := \sL \cup \{L_{r+1}\}$ be the line configuration obtained from  $\mathcal{L}$ by adding the line $L_{r+1}$.  Then the equisingular deformations of $X_{\mathcal{L'}}(n)$ are infinitesimally trivial.
\end{corollary}

\begin{remark}
The asymptotic Chern ratio $c_1^2 / c_2$  of the HK covering branched on an infinitesimally rigid line configuration $\mathcal{L}$ is strictly  greater than two.
\end{remark}

\section{Embeddings of Del Pezzo surfaces in Cartesian products of  the projective line}

The simplest Del Pezzo surfaces $S$ (surfaces with $-K_S$ ample) are those of highest anticanonical degree $d$ ($d : = K^2_S$),  $\PP^2$ and  $\PP^1 \times \PP^1$
($d=9$, respectively $d=8$).

The other Del Pezzo surfaces of degree $d$ are blow-ups of $\PP^2$ in $k : = 9-d$ points $P_1, \dots, P_k$, none infinitely near,
such that no three lie on a line, no six lie on a conic.

For $d=8$, $S_8$ is the closure of the rational map $\phi_1 : \PP^2 \setminus \{ P_1\} \ra \PP^1$, and if we choose coordinates so that $P_1 = (0,0,1)$, then:

$$ S_8 \subset   \PP^2 \times  \PP^1, \  S_8 : = \{ (x_1,x_2,x_3) (y_1, y_2) | x_1 y_2 = x_2 y_1 \}.$$

We can argue similarly for the blow up of more than one point. For instance, if we choose coordinates for which $P_2 = (0,1,0)$, then: 

$$ S_7 \subset   \PP^2 \times  \PP^1 \times  \PP^1, \  S_7 : = \{ (x_1,x_2,x_3) (y_1, y_2) (z_1, z_2)| x_1 y_2 = x_2 y_1 , \ x_1 z_1 = z_2 x_3 \}.$$

We observe that here $S_7$ is a complete intersection of type $(1,1,0), (1,0,1)$ hence its anticanonical divisor is induced by the divisor of
multidegree $(1,1,1)$.

 In particular, the anticanonical embedding of $S_7$ is obtained by cutting the Segre product $ \PP^2 \times  \PP^1 \times  \PP^1\subset \PP^{11}$
with four hyperplane sections.

Similarly, the Del Pezzo surface of degree $d$ is a complete intersection inside $ \PP^2 \times  (\PP^1 )^k$, and its anticanonical divisor
has multidegree $(3-k,1,1, \dots,1)$.

Hence we see that, for $k\geq 3 \Leftrightarrow d = 9 -k \leq 6$, $S_{9-k}$ embeds in $ (\PP^1 )^k$.

The first instance is the Del Pezzo of degree $d=6$, where we may choose coordinates such that  $P_3 = (1,0,0)$.

$$ S_6 \subset   \PP^2 \times ( \PP^1)^3 ,$$ $$ \  S_6 : = \{ (x_1,x_2,x_3) (y_1, y_2) (z_1, z_2) (t_1, t_2)| x_1 y_2 = x_2 y_1 , \ x_1 z_1 =  x_3 z_2, \ x_2 t_2= x_3 t_1 \}.$$

The embedding in $( \PP^1)^3$ is easily obtained eliminating the variable $x: = (x_1, x_2, x_3)$, in that we view the above equation as saying that 
there exists a non zero vector $(x_1,x_2,x_3) \in \ker (B)$, where $B$ is the matrix
\begin{equation}
 B:=
 \begin{pmatrix}
 y_2 & -y_1 & 0\\
  z_1  &0  & - z_2\\
 0 & t_2  &- t_1
 \end{pmatrix} .
\end{equation}

The immediate conclusion is that the Del Pezzo is the hypersurface  where $ det (B)$ vanishes:
$$ S_6 \subset   ( \PP^1)^3 \  S_6 : = \{  (y_1, y_2) (z_1, z_2) (t_1, t_2)| y_2 z_2 t_2  = y_1 z_1 t_1\}.$$

Going further, for the Del Pezzo of degree five, we may choose coordinates so that the point $P_4 = (1,1,1)$ hence:

$$ S_5 \subset   ( \PP^1)^4 \  S_5 : = \{  (y_1, y_2) (z_1, z_2) (t_1, t_2) (w_1,w_2)|  \rank (A) = 2 \},$$
where $A$ is the matrix
\begin{equation}
 A:=
 \begin{pmatrix}
 y_2 & -y_1 & 0\\
  z_1  &0  & - z_2\\
 0 & t_2  &- t_1\\
 -w_2& w_1 + w_2 & - w_1
 \end{pmatrix} .
\end{equation}

Hence $S_5$ is the locus of zeroes of the four  $(3\times 3)$-minors of $A$, and the resolution of 
its ideal sheaf is the classical Hilbert-Burch resolution, as we shall see in the  sequel in more detail.

Finally, in the case of five points, the point $P_5$ will be the common zero of two linear forms
$$ P_5 =  \{ x_1 - \la x_2 = x_3 - \mu x_2=0\}. $$ 

Accordingly, 

$$ S_4 \subset   ( \PP^1)^5 \  S_5 : = \{  (y_1, y_2) (z_1, z_2) (t_1, t_2) (w_1,w_2)(v_1, v_2)|  \rank (A) = 2 \},$$
where $A$ is the matrix
\begin{equation}
 A:=
 \begin{pmatrix}
 y_2 & -y_1 & 0\\
  z_1  &0  & - z_2\\
 0 & t_2  &- t_1\\
 -w_2& w_1 + w_2 & - w_1\\
 v_2 & -\la v_2 -\mu v_1 & v_1
 \end{pmatrix} .
\end{equation}

Hence any such Del Pezzo surface $S_4$ of degree $4$ (this time we have a two-dimensional family)
is the locus of zeroes of the $10$   $(3\times 3)$-minors of $A$, and the resolution of 
its ideal sheaf is provided by the Eagon-Northcott complex (a generalization of the Hilbert-Burch resolution).

Similarly each Del Pezzo surface of degree $9-k$, for $k \geq 3$, is the determinantal variety in $P : = (\PP^1)^k$
associated to a $k \times 3$-matrix
$A$: $S_{9-k} $ is the locus of points where the rank of $A$ is $2$, and its ideal sheaf has the following 
Eagon-Northcott resolution.

 We define here $\sL_i$ to be the pull-back of $\hol_{\PP^1} (1)$ via the ith-projection, 
 we take $U$ a 3-dimensional vector space and we set $\sE$ to be  the locally free sheaf
$$ \sE : =  \oplus_1^k \sL_i^{-1}.$$

Then the matrix $A$ yields a linear map $$ A : \sE \ra \sU : = U \otimes \hol_P.$$

Then the Eagon-Northcott complex
$$ 0 \ra  (\wedge ^k  \sE)  \otimes (S^{k-3}(\sU))^{\vee} \ra \dots (\wedge ^4 \sE)  \otimes (\sU)^{\vee}     \ra  \wedge ^3 \sE   \ra  \wedge ^3 \sU \cong \hol_P ,$$
where the last homomorphism is given by $ \wedge ^3 A$,  is a resolution of
the ideal sheaf of $S_{9-k} $.

In the next section we shall specialize to the case of $S_5$, where the Eagon-Northcott complex reduces to the length two Hilbert-Burch complex
$$ 0 \ra (\wedge ^4 \sE)  \otimes (\sU)^{\vee} \cong \hol_P ( -\sum_j L_j) ^3     \ra  \wedge ^3 \sE  \cong \oplus_{i=1}^4 \hol_P (L_i - \sum_j L_j  )  \ra  \wedge ^3 \sU \cong \hol_P ,$$
where the first  homomorphism is given by the matrix $A$, the second by the matrix $ \wedge ^3 A$.

\section{Embeddings of the Del Pezzo surface of degree 5 in $(\PP^1)^4$ and $(\PP^1)^5$}

We  recall some notation introduced in \cite{takagi}, as well as   some  intermediate results established there.

 The Del Pezzo surface $Y: = Y_5$ of degree $5$ is the blow-up  of the plane in the $4$ points $P_1, \dots, P_4$
  of a projective basis.

The Del Pezzo surface $Y$  is indeed  the moduli space of ordered quintuples of points in $\PP^1$, 
and its  automorphism group  is isomorphic to $\mathfrak S_5$.

The obvious action of the symmetric group $\mathfrak S_4$ permuting the $4$ points extends in fact to an action
of the  symmetric group $\mathfrak S_5$. 

This can be seen as follows. The six lines in the plane joining pairs $P_i, P_j$ can be labelled as $L_{i,j} $, with $ i,j \in \{1, 2,3,4\}, i \neq  j $.

Denote by $E_{i,5}$ the exceptional curve lying over
the point $p_i$, and denote, for $i \neq   j \in \{1, 2,3,4\}$, by $E_{h,k} = E_{k,h}$ the strict transform in $Y$ of the  line $L_{i,j}$, if $\{1, 2,3,4\}=  \{i, j, h,k\}$.
For each choice of $3$ of the four points, $\{1, 2,3,4\} \setminus \{h\}$, consider the standard Cremona transformation $\s_h$
based on these three points. To it we associate the transposition $(h,5) \in \mathfrak S_5$, and the upshot is that
$\s_h$ transforms the $10$ $(-1)$ curves $E_{i,j}$ via the action of $(h,5) $ on pairs of elements in $\{1, 2,3,4,5\}$. 

There are  five geometric objects permuted by $\mathfrak S_5$: namely, $5$ fibrations $\varphi_i : Y \ra \PP^1$,
induced, for $1 \leq i \leq 4$, by the projection with centre $P_i$, and, for $i=5$, by the pencil of conics through
the $4$ points. Each fibration  is a conic bundle, with exactly three singular fibres, correponding to the possible partitions of
type $(2,2)$ of the set $\{1, 2,3,4,5\} \setminus \{i\}$.

The intersection pattern of the curves $E_{i,j}$, which generate the Picard group of $Y$
is dictated by the simple rule (recall that $ E_{i,j}^2 = -1, \ \forall i \neq j$)
$$ E_{i,j} \cdot E_{h,k} = 1 \ \Leftrightarrow \{i,j\} \cap \{h,k\}= \emptyset, \ \ E_{i,j} \cdot E_{h,k} = 0   \Leftrightarrow \{i,j\} \cap \{h,k\}\neq  \emptyset, \{i,j\}.$$
In this picture the three singular fibres of $\varphi_1$ are $$E_{3,4} +  E_{2,5}, \ E_{2,4} +  E_{3,5}, \ E_{2,3} +  E_{4,5}.$$

The relations among the $E_{i,j}$'s in the Picard group come from the linear equivalences 
$E_{3,4} +  E_{2,5} \equiv E_{2,4} +  E_{3,5} \equiv  E_{2,3} +  E_{4,5}$ and their $\mathfrak S_5$-orbits.

An important observation is that  $Y$ contains exactly ten lines, i.e. 
irreducible curves $E$ with $ E^2 = E K_Y = -1$.

\begin{theorem}
 $Y$ embeds into $(\PP^1)^4$ via $\varphi_1 \times \dots \times  \varphi_4$
 and in $(\PP^1)^5$ via $\varphi_1 \times \dots \times \varphi_5$.
 \end{theorem}
 
  \begin{proof}
 Obviously it suffices to show the first assertion.  In turn, the first assertion was proven in the previous section.

 \end{proof}
 
 \begin{theorem}\label{eqdp}
 Let $\Sigma \subset (\PP^1)^4=:Q$, with coordinates 
 $$(v_1:v_2),(w_1:w_2),(z_1:z_2), (t_1:t_2),$$
  be the image of the Del Pezzo surface $Y$ via  $\varphi_1 \times \dots \times  \varphi_4$.
 Then the equations of $\Sigma$ are given by the four $3 \times 3$-minors of the following Hilbert-Burch matrix:
 \begin{equation}
 A:=
 \begin{pmatrix}
 t_2 & -t_1 &t_1+t_2\\
  v_1 & v_2 & 0\\
 w_2 & 0 &w_1 \\
 0&-z_1&z_2 
 \end{pmatrix} .
\end{equation}
In particular, we have a Hilbert-Burch resolution:
\begin{equation}
0 \ra (\hol_Q(-\sum_{i=1}^4 H_i))^{\oplus 3} \ra \bigoplus_{j=1}^4( \hol_Q(-\sum_{i=1}^4 H_i + H_j)) \ra \hol_Q \ra \hol_{\Sigma} \ra 0,
\end{equation}
where $H_i$ is the pullback to $Q$ of a point in $\PP^1$ under the i-th projection, and where the first two homomorphisms are given by the respective matrices $A$, $\wedge^3(A)$.
 \end{theorem}
 
  \begin{proof}
  The proof was essentially already given in the previous section, but we repeat it here with more details to keep track of the symmetries of the Del Pezzo surface.
 
 To this purpose, view the vector space $V$ such that  $\PP^2$ is the projective space of lines in $V$
 as the $\mathfrak S_4$-representation generated by vectors $e_1, e_2, e_3, e_4$ such that
 $$ e_1+  e_2 + e_3 +  e_4= 0.$$
 
 In other words, $e_4$ is the vector $(-1,-1,-1)$.
 
 Introduce linear forms $x_1, x_2, x_3, y_1, y_2, y_3$ such that  $x_1, x_2, x_3$ is the dual basis of $e_1, e_2, e_3$,
 while the $y_i$'s are determined by the following equations (the last being a consequence of the previous three).
\begin{eqnarray}\label{pencils}
 x_1-x_2+y_3 = 0, &  x_2-x_3+y_1 = 0, & x_3-x_1+y_2 = 0, \\
 y_1+y_2+y_3 = 0. 
 \end{eqnarray}
 
 We get 4 equations, for each possible coordinate projection $(\PP^1)^4 \ra (\PP^1)^3$.
 
 We have seen that  $ (\varphi_3 \times  \varphi_1  \times  \varphi_2) (x_1: x_2:x_3) = ((x_1:x_2),(x_2:x_3),(x_3:x_1))$,

 and this leads to the equation 
  $$
  G_1:=v_1w_1z_1 - v_2w_2z_2 = 0.
  $$
 This corresponds to the triangle with vertices $e_1,e_2,e_3$ and with sides 
 $$
 \{x_1=0\}, \{x_2=0\},\{x_3=0\}.
 $$
 
 Similarly,  for the triangle with vertices $e_1,e_2,e_4$, whose   sides are the zero sets  of the linear forms $y_1,x_3,y_2$,   since  $(v_1,v_2)=(x_1, x_2)$, $(w_1,w_2)=(x_2, x_3)$,
$(z_1,z_2) =(x_3,x_1)$, 
we obtain a map to $(\PP^1)^3$ given by
$$
((y_1:x_3),(x_3:y_2),(y_2:y_1)) =(w_2-w_1:w_2),(z_1:z_2-z_1),(t_2:t_1)),
$$
hence an equation:
$$
G_2:=(w_2-w_1)z_1t_2 - w_2(z_2-z_1)t_1=0.
$$
The triangle with vertices $e_2,e_3,e_4$ and  with sides given by the vanishing of the linear forms $y_2,x_1,y_3$ similarly yields the equation:
$$
G_3:=(z_2-z_1)v_1(t_1+t_2) +z_2(v_2-v_1)t_2 = 0.
$$
Finally, the triangle with vertices   $e_3,e_1,e_4$ and  with sides given by the vanishing of the linear forms $y_3,x_2,y_1$  yields the equation:
$$
G_4:=(v_2-v_1)w_1t_1 +v_2(w_2-w_1)(t_1+t_2) = 0.
$$
A direct calculation shows that the four equations $G_1, \ldots ,G_4$ are easily seen to be the entries of $\wedge^3 (A)$, and that
the linear syzygies among the four equations are given by the columns of the matrix $A$.

Therefore we have a complex 
\begin{equation*}
0 \ra (\hol_Q(-\sum_{i=1}^4 H_i))^{\oplus 3} \ra \bigoplus_{j=1}^4( \hol_Q(-\sum_{i=1}^4 H_i + H_j)) \ra \hol_Q \ra \hol_{\Sigma} \ra 0
\end{equation*}
of Hilbert-Burch type. Hence $G_1,G_2,G_3,G_4$ define the smooth subscheme $\Sigma \subset Q$, and the  exactness of the above complex at the final two terms follows. By the Buchsbaum-Eisenbud criterion (\cite{exactcomplex}) the whole sequence is exact.

  \end{proof}
  
  \begin{remark}
  A consequence of our result is that the multicone $\sC \Sigma \subset \CC^8$ over $\Sigma$ is Cohen-Macaulay, from which it follows that it must have a Hilbert-Burch resolution. We could also have proven this directly using projective plurinormality and vanishing of cohomology groups, but our interest was more concerned in finding the matrix $A$ explicitly.
  \end{remark}
  
  \begin{remark}
   $\sC \Sigma$ is Cohen-Macaulay, but not Gorenstein, since 
   $$
   K_{\Sigma} =K_Y = -3L + \sum_{i=1}^4 E_{i,5} = -\sum_{i=1}^4 H_i +L,
   $$
   where $L$ is the pull-back of a line in $\PP^2$.
  \end {remark}
 \begin{remark}\label{p15}
 We shall show in a future paper that the equations of 
 $$
 Y':=(\varphi_1 \times \ldots \times \varphi_5)(Y) \subset (\PP^1)^5
 $$
 are the ten equations obtained by the ten coordinate projections $P:=(\PP^1)^5 \ra (\PP^1)^3$ and that we have an exact sequence 
 \begin{multline}
 0 \ra (\hol_P(-\sum_{i=1}^5 H_i))^{\oplus 6} \ra (\bigoplus_{j=1}^5\hol_P(-\sum_{i=1}^5 H_i + H_j)) ^{\oplus 3} \ra \\
 \ra \bigoplus_{h<k}( \hol_P(-\sum_{i=1}^5 H_i + H_k +H_h)) \ra \hol_P \ra \hol_{Y'} \ra 0,
\end{multline}
where the first syzygies are the pull-backs of the syzygies obtained for each projection $(\PP^1)^5 \ra (\PP^1)^4$. 

 Observe that the shape of this resolution is the same as the Eagon-Northcott complex for a Del Pezzo surface $S_4$ of degree 4.

But if this resolution were associated to a $5 \times 3$ matrix of the same type, then we would get a Del Pezzo surface of degree 4
and not of degree 5. 
\end{remark}

\section{Embedding of the surfaces $HK (n, \sC \sQ)$ into products of Fermat curves}

 In this section we shall show that   $S : = HK (n, \sC \sQ)$ embeds  into $C(n)^4$, respectively $C(n)^5$, where 
 $C(n)$ is the Fermat curve of degree $n$, $$C(n) = \{ Y_1^n + Y_2^n + Y_3^n = 0\} \subset \PP^2.$$
 $C(n)$ is a natural $(\ZZ/n)^2$ Galois-covering of 
 $$\PP^1 = \{ (y_1: y_2:y_3) \in \PP^2 \ |  \ y_1 + y_2 + y_3=0\}.$$
 
 We shall see that this embedding  is obtained via (the normalization of) the pull back of the embedding of the Del Pezzo surface $Y$ in $(\PP^1)^4$ (respectively in $(\PP^1)^5$).
 
 Observe first that each pencil $\varphi_i \colon Y \ra \PP^1$ can be rewritten as
 $$
 \varphi_i \colon Y \ra \Lambda_i \subset \PP^2,
 $$
 where $\Lambda_1, \Lambda_2, \Lambda_3, \Lambda_4$ are the lines in $\PP^2$ defined by
 equations which are consequences of  the following equations (\ref{pencils}): 
  \begin{eqnarray*}
 x_1-x_2+y_3 = 0, \\  
 x_2-x_3+y_1 = 0, \\ 
  x_3-x_1+y_2 = 0, \\
  y_1+y_2+y_3 = 0. 
 \end{eqnarray*}
 
 Hence we define
   \begin{eqnarray*}
 \Lambda_3: = & \{ v_1-v_2+v_3 = 0\}, \\  
 \Lambda_1: =& \{ w_1-w_2+w_3 = 0\}, \\ 
 \Lambda_2: =& \{  z_1-z_2+z_3 = 0 \}, \\
 \Lambda_4: = & \{   t_1+t_2+t_3  = 0\}. 
 \end{eqnarray*}
 
 The map  $\varphi_3 \times \varphi_1 \times \varphi_2 \times \varphi_4 $ is then expressed by: 
 \begin{eqnarray*}
(v_1,v_2,v_3):=(x_1,x_2,y_3),  & (w_1,w_2,w_3):=(x_2,x_3,y_1),\\  
(z_1,z_2,z_3):=(x_3,x_1,y_2),  & (t_1,t_2,t_3):=(y_1,y_2,y_3).
 \end{eqnarray*}
 
 The equations of the image $\Sigma \subset (\PP^1)^4$ of the Del Pezzo surface $Y$ now read out in
 a more symmetric way  as:
  \begin{eqnarray*}
v_1w_1z_1 - v_2w_2z_2 = 0, \\  
w_3z_1t_2 - w_2z_3t_1 = 0, \\  
z_3v_1t_3 - z_2v_3t_2 = 0, \\  
v_3w_1t_1 - v_2w_3t_3 = 0.
 \end{eqnarray*}
 Each line $\Lambda_i$ has a $(\ZZ/n \ZZ)^2$ cover isomorphic to the Fermat curve $C(n)_i$. Indeed, we get 
  \begin{eqnarray*}
 (V_i^n=v_i ), & C(n)_3 = \{ V_1^n-V_2^n+V_3^n = 0\}, \\  
 (W_i^n=w_i ), &  C(n)_1 = \{W_1^n-W_2^n+W_3^n = 0\}, \\  
  (Z_i^n=z_i ), &  C(n)_2 = \{Z_1^n-Z_2^n+Z_3^n = 0\}, \\  
 (T_i^n=t_i ), & C(n)_4 =  \{T_1^n-T_2^n+T_3^n = 0\}.
 \end{eqnarray*}
 
 The pull-back of the four equations defining $\Sigma$ split into $n^4$ equations:
 
  \begin{eqnarray*}
V_1W_1Z_1 = \epsilon_1 V_2W_2Z_2, \\  
W_3Z_1T_2 = \epsilon_2W_2Z_3T_1, \\  
Z_3V_1T_3 = \epsilon_3Z_2V_3T_2 , \\  
V_3W_1T_1 = \epsilon_4V_2W_3T_3,
 \end{eqnarray*}
 where $\epsilon_i$, $1 \leq i \leq 4$ is a $n$-th root of unity. 
 
 Moreover, taking the product of the above four equations we get:
 $$
 V_1V_2V_3W_1W_2W_3Z_1Z_2Z_3T_1T_1T_3 = \epsilon_1\epsilon_2\epsilon_3\epsilon_4V_1V_2V_3W_1W_2W_3Z_1Z_2Z_3T_1T_1T_3,
 $$
 whence we get $\epsilon_1\epsilon_2\epsilon_3\epsilon_4 = 1$.
 
 Therefore the normalization of the fibre product $Y \times _{(\PP^1)^4} C(n)^4$ splits into $n^3$ isomorphic
 components. 
 
 The Hirzebruch-Kummer covering $S = HK(n,\sC\sQ)$ admits a natural morphism to the fibre product
 because the covering $S \ra Y$, with Galois group $(\ZZ/n)^5$,  is obtained taking the $n$-th roots of $x_1,x_2,x_3,y_1,y_2,y_3$.
 
 Since the Galois group of $C(n)^4 \ra (\PP^1)^4$ is $(\ZZ/n)^8$, $S$ maps birationally onto
 one of the components of the fibre product, without loss of generality the one where
 we pick $\epsilon_1= \epsilon_2= \epsilon_3=\epsilon_4=1$, that we shall denote by $S_n$. 
 
 We summarize the situation via the following diagram:
 
\begin{equation}
\xymatrix{
 S \ar[r] &S_n \ar_f[d]\ar^{(\ZZ/n)^5}[r]&Y\ar^{\varphi_1 \times \ldots \times \varphi_4}[d]\\
& C(n)^4\ar_{(\ZZ/n)^8}[r]&(\PP^1)^4
} .
\end{equation}

In order to show that we have an isomorphism $ S \cong S_n$, it suffices to show that $S_n$ is smooth,
because a finite and birational morphism between normal varieties is an isomorphism (it is also possible to
prove directly that $S$ embeds).

The smoothness of $S_n$ follows right away from the 
 following equations for $S_n$ in the product of four Fermat curves:
 
\begin{theorem}\label{eqhk}
The equations of $S_n = HK(n,\sC\sQ) \subset C(n)^4$ are given by the four $3 \times 3$-minors of the following  matrix:
 \begin{equation}
 A':=
 \begin{pmatrix}
 T_2 & -T_1 &-T_3\\
  V_1 & V_2 & 0\\
 W_2 & 0 &W_1 \\
 0&-Z_1&Z_2 
 \end{pmatrix} ,
\end{equation}
and the linear syzygies among the four equations are given by the columns of the matrix $A'$.
\end{theorem}

\begin{remark}
Similarly (cf. Remark \ref{p15})  one can find  the equations of the HK-covering
 $S \subset C(n)^5$. 
 We get then ten equations which are obtained as before from  the pullback of the ten equations of 
 $$
 Y':=(\varphi_1 \times \ldots \times \varphi_5)(Y) \subset (\PP^1)^5.
 $$
\end{remark}

\end{document}